\documentclass[11pt]{amsart}
\usepackage[margin=1.5in]{geometry}

\usepackage{xcolor}
\usepackage{amssymb}

\newtheorem{theorem}{Theorem}[section]
\newtheorem{lemma}[theorem]{Lemma}
\theoremstyle{definition}
\newtheorem{definition}[theorem]{Definition}
\newtheorem{example}[theorem]{Example}

\theoremstyle{remark}
\newtheorem{remark}[theorem]{Remark}

\numberwithin{equation}{section}

\newcommand{\nats}{\mathbb{N}}
\newcommand{\ints}{\mathbb{Z}}
\newcommand{\reals}{\mathbb{R}}
\newcommand{\xp}{\mathbb{E}}

\newcommand{\complex}{\mathbb{C}}
\newcommand{\el}{\mathcal{L}}
\newcommand{\cupdot}{\mathbin{\mathaccent\cdot\cup}}

\DeclareMathOperator{\sgn}{sgn}

\begin{document}

\title[A New Proof of the Abstract Random Tensor Estimate] {A New Proof of the Abstract Random Tensor Estimate by Deng, Nahmod, and Yue}

\author{Claire Kaneshiro}
\address{Claire Kaneshiro, Department of Mathematics, Princeton University Princeton, NJ 08544}
\email{clairekaneshiro@princeton.edu}
\subjclass[2020]{60B20, 15B52, 33C45, 35Q55}
\date{\today}

\begin{abstract}
We provide a new proof of the abstract random tensor estimate. This estimate was initially proven by Deng, Nahmod, and Yue (2022) using the moment method. The key new tool in our proof is the direct use of the non-commutative Khintchine inequality with the probabilistic decoupling of the product of Gaussians. Hermite and generalized Laguerre-type polynomials allow us to account for pairings in the real and complex-valued Gaussians, respectively, and remove the square-free (tetrahedral) requirement. 
\end{abstract}

\maketitle
\section{Introduction}
The non-commutative Khintchine inequality is an important estimate in random matrix theory \cite[Theorem 3.2 and Corollary 3.3]{vH17}. For $\mathbb{F} \in \{\reals, \complex \}$, let $X \in \mathbb{F}^{n\times n}$ be a random matrix whose entries are centered and jointly Gaussian. Then $X$ can be written as $X = \sum_{k=1}^s g_k A_k$, where $(g_k)_{k=1}^s$ is a sequence of independent standard Gaussians and $(A_k)_{k=1}^s \subset \mathbb{F}^{n\times n}$ are the coefficient matrices. The non-commutative Khintchine inequality states that
\begin{equation} \label{eq: noncommutative Khintchine Intro}
    \xp\Bigg[\bigg\| \sum_{k=1}^s g_k A_k \bigg\|_{\text{op}}^p \Bigg]^{\frac{1}{p}} \lesssim \sqrt{ p \log n} \max \Bigg( \bigg\|\sum_{k=1}^s A_kA_k^*  \bigg\|_{\text{op}} ^{\frac{1}{2}}, \bigg\|\sum_{k=1}^s A_k^* A_k  \bigg\|_{\text{op}} ^{\frac{1}{2}}\Bigg).
\end{equation}
Effectively, this allows us to bound the operator norm of random matrix by its underlying covariance structure, with only logarithmic loss in terms of dimension. For more precise discussion of this statement see Lemma~\ref{NoncommutativeKhintchine}.
Recently, Deng, Nahmod, and Yue \cite[Proposition 2.8 and Proposition 4.14]{DNY20} proved the abstract random tensor estimate, which is a higher-order generalization of the non-commutative Khintchine inequality.
Fundamentally, the random tensor estimate replaces the Gaussian $g_k$ with the product of $m$ Gaussians $g_{n_1}g_{n_2}...g_{n_m}$ and replaces the coefficient matrices $A_k$ by tensors (defined below). 

\begin{definition}[Tensors]
    A tensor is a map $h:(\ints^d)^J \rightarrow \reals$, where $J$ is a finite set. If $J = \{j_1, ..., j_k\}$, we often write $h=h_{n_J} = h_{n_{j_1}...n_{j_k}}$, where $n_J = (n_j : j \in J)$ are the input variables.
\end{definition}
When the basis is fixed, the tensor can be viewed as a higher-dimensional matrix. 

\begin{definition} [Tensor norm] \label{def: Tensor Norm}
    Let $h = h_{n_J}$ be a tensor. We say $X$ and $Y$ form a partition of $J$ if $X \cup Y = J$ and $X \cap Y = \emptyset$. We denote this by $X \cupdot Y = J$. Then, for such $X,Y$ we define the tensor norm $\| \cdotp \|_{n_X\rightarrow n_Y}$ as follows
    \begin{equation} \label{eq: Tensor Norm}
        \|h\|_{n_X\rightarrow n_Y}^2 := \text{sup} \Bigg\{ \sum_{n_Y} \bigg| \sum_{n_X} h_{n_J}v_{n_X}\bigg|^2: \sum_{n_X\in (\ints^d)^X} \Big|v_{n_X}\Big|^2 \leq 1 \Bigg\}.
    \end{equation}

\end{definition}

We further remark that, even for a fixed index set $J$, the partition $X \cupdot Y$ is not unique. For distinct partitions into $X$ and $Y$, we end up with distinct tensor norms. However, if we fix $X$ and $Y$, we can view $h$ as a linear operator from $\ell_{n_X}^2$ to $\ell_{n_Y}^2$, where $l^2_{n_Z}$ denotes the space of square-summable sequences $(x_{n_Z})_{n_Z \in (\ints^d)^Z}$, for $Z \in \{X,Y\}$. 
In this case, the tensor norm coincides with the usual operator norm. \\

The objective of this paper is to provide a new simple proof the abstract random tensor estimate (for the original proof by Deng-Nahmod-Yue we refer to \cite[Propositions 2.8 and 4.14]{DNY20}). Technically, we state a slightly different version of the theorem than in \cite{DNY20}, since we drop the square-free condition. 

\begin{theorem} [\protect{Abstract random tensor estimate \cite[Proposition 4.14]{DNY20}}] \label{Abstract Random Tensor Estimate}
Fix $k \in \nats$. Let $h = h_{n_Jn_An_B}$ be a deterministic tensor, where $A$ and $B$ are finite index sets and $J = \{1, 2, \dots, k\}$. Let $n_J = (n_1, n_2, ..., n_k) \in (\ints^d)^k$ and $N \in 2^{\nats}$, such that 
\begin{equation*}
    |n_1|, ..., |n_k|, \max_{a \in A}|n_a|, \max_{b \in B}|n_b| \leq N 
\end{equation*}
on the support of $h$, where $|\cdot|$ is the standard $\ell_1$ norm. Functionally, the number $N$ is a rough bound on the size of the support.

Furthermore, let $(g_n)_{n \in \ints^d}$ be a sequence of pairwise independent, standard complex-valued Gaussians. Fix the signs $\{\iota_1, \iota_2, \dots, \iota_k\} \in \{\pm 1\}^k$, with $g^{\iota}$ define by $g^{+1}=g$ and $g^{-1}=\overline{g}$. Then let $G = G_{n_An_B}$ be the random tensor defined as 
  \begin{equation*}
        G_{n_An_B} := \sum_{n_J \in (\ints^d)^k} h_{n_Jn_An_B} :\prod_{j\in J} g_{n_j}^{\iota_j}:\ ,
    \end{equation*}
    where $:\prod_{j\in J} g_{n_j}^{\iota_j}:$ is the renormalization of the product $\prod_{j\in J} g_{n_j}^{\iota_j}$, which will be defined using Laguerre polynomials (see Section~\ref{section: Laguerre}). 
Then it holds for all $p \geq 1$ that 
    \begin{equation}
       \xp[\|G\|^p_{n_A\rightarrow n_B}]^{\frac{1}{p}} \leq C p^{\frac{k}{2}} (\log N)^{\frac{k}{2}} \max_{X \cupdot Y = {J}} \|h\|_{n_An_X \rightarrow n_Bn_Y},
    \end{equation}
where $C=C(k,d,|A|,|B|)$ is a constant depending on the number of Gaussians in the product $k$, the dimension $d$, and the cardinality of the index sets $A$ and $B$.
\end{theorem}
Fundamentally, we bound (in expectation) the spectral norm of the random tensor in terms of the underlying deterministic structure. 
We will discuss other differences below, but one key new component of our proof is this renormalization, which allows us to account for pairings. We first recall the definition of a pairing. 
\begin{definition}
  With the same assumptions as Theorem~\ref{Abstract Random Tensor Estimate}. For $n_1, n_2 \in \ints^d$ with corresponding signs $\iota_1,\iota_2 \in \{\pm 1\}$, we say $(g_{n_1}, g_{n_2})$ form a \textit{pairing} if $\iota_1 + \iota_2 = 0$ and $n_1 = n_2$. In the real-valued case, $(g_{n_1},g_{n_2})$ form a pairing if $n_1 = n_2$. 
\end{definition}

Note that if $g_{n_1}^{\iota_1}$ and $g_{n_2}^{\iota_2}$ form a pairing, then the two Gaussians are not independent and the expectation of the product $g_{n_1}^{\iota_1}g_{n_2}^{\iota_2}$ is not centered. Deng, Nahmod, and Yue \cite{DNY20} require that there are no pairings (the product is square-free or tetrahedral) in $n_J$ (on the support of the deterministic tensor). Therefore, the product does not require normalization. 
We give an explicit construction in terms of Hermite and Laguerre-type polynomials to account for the pairings in real-valued and complex-valued, respectively (see Section 3 for details).

\begin{example}
    If we let $h = h_{abcd}$ be a deterministic tensor, where $a,b,c,d$ are finite indices with size at most $N$. Let $(g_a)_{a=1}^n$ and $(g_b)_{b = 1}^m$ be sets of pairwise independent complex-valued Gaussians, where $m$,$n$ $\leq N$. Then, it holds that 
     \begin{align*}
        & \Big\|\sum_{a, b} h_{abcd} (g_a\overline{g_b} -\delta_{ab}) \Big\|_{c \rightarrow d} 
         \\ &\leq C p (\log N) \max\left\{ \|h\|_{abc \rightarrow d}, \|h\|_{ac \rightarrow bd}, \|h\|_{bc \rightarrow ad}, \|h\|_{c \rightarrow abd} \right\},
     \end{align*}
     where $C$ is a constant and $(g_a\overline{g_b} -\delta_{ab})$ is the renormalization of $g_a\overline{g_b}$.
\end{example}

The abstract random tensor estimate was initially used by Deng-Nahmod-Yue \cite{DNY20} for power-type nonlinear Schr\"odinger equations. Since then it has been used in other dispersive PDE contexts; see also Bringmann \cite{B24}, Oh- Wang-Zine \cite{OWZ22}, and Deng-Nahmod-Yue \cite{DNY21}. In fact, Bringmann-Deng-Nahmod-Yue \cite{BDNY24} proved a bilinear random tensor estimate in the case where there is one Gaussian, instead of a product of $k$ Gaussians. It may be interesting to see whether our argument can be used for variants of \cite[Lemma 8.1]{BDNY24}.

The original proof of the abstract random tensor in \cite{DNY20} uses the moment-method (see \cite{vH17} Section 2.2). Whereas our proof uses the non-commutative Khintchine inequality directly, by decoupling the random variables.
Our proof is also more modular and depends on the following four tools:
\begin{enumerate}
    \item [(i)] Laguerre polynomials (in Section 3)
    \item [(ii)] Tensor merging estimate (Lemma~\ref{lemma: MergingEstimate}, in Section 2.1)
    \item [(iii)] Gaussian case (Lemma~\ref{lemma: GaussianCase}, in Section 4)
    \item [(iv)] Probabilistic decoupling (Lemma~\ref{lemma: Probabilistic Decoupling}, in Section 4)
\end{enumerate}

\begin{remark}
This estimate is motivated by applications to PDEs and we present this proof in the language of this field. Though, some of the tools we use have been known implicitly or explicitly in the free probability community.
References to higher order noncommutative Khintchine inequalities, such as in Theorem \ref{Abstract Random Tensor Estimate}, date back to the work of Haagerup and Pisier, see \cite[Remark 9.8.9]{P03} and \cite{HP93}.
Also, it is a known result that Hermite polynomials satisfy a non-square-free decoupling inequality (see, for example, \cite[Section 2]{AG93}).
Moreover, the construction of the random tensor in Theorem~\ref{Abstract Random Tensor Estimate} is similar to the matrix chaos studied by Bandeira-Lucca-Nizi\'c-Nikolac-van Handel  \cite[Section 2]{BLNvH24}. However, this model requires the square-free condition, whereas our argument can account for the pairings. 
\end{remark}
\vspace{5mm}

\section{Background}

\subsection{Deterministic background}

\begin{lemma} [Tensor norm duality]
    Let $h = h_{n_J}$ be a tensor and let $X$ and $Y$ be a partition of $J$.
    We observe that the square root of \eqref{eq: Tensor Norm} is an $\ell_2$-based operator norm of the tensor $(\sum_{n_X} h_{n_J}v_{n_X})$. Therefore, by duality, we also have that 
     \begin{equation} \label{eq: Tensor Norm Duality}
        \|h\|_{n_X\rightarrow n_Y} := \text{sup} \Big\{ \Big| \!\sum_{n_X, n_Y}\!  \! h_{n_J}  v_{n_X} w_{n_Y} \Big| : \! \! \! \sum_{n_X\in (\ints^d)^X} \! |v_{n_X}|^2,  \! \! \! \sum_{n_Y\in (\ints^d)^Y} \! |w_{n_Y}|^2 \leq 1 \Big\},
    \end{equation}

    which implies that 
    \begin{equation*} 
    \| \overline{h}\|_{n_X \rightarrow n_Y}=  \| h\|_{n_X \rightarrow n_Y} = \| h\|_{n_Y \rightarrow n_X}.
    \end{equation*} 
\end{lemma}

\begin{lemma}[\protect{Merging estimate \cite[Lemma 2.5]{DNY20}}] 
\label{lemma: MergingEstimate}
Let $A_1, A_2, B_1, B_2,$ and $C$ be disjoint finite index sets and let $h^{(1)} = h^{(1)}_{n_{A_1} n_{B_1} n_{C}}$ and $h^{(2)} = h^{(2)}_{n_{A_2} n_{B_2} n_{C}}$ be two different tensors. Then, it holds that 
\begin{align*}
    &\Big\| \!\sum_{n_C} h^{(1)}_{n_{A_1} n_{B_1} n_{C}} h^{(2)}_{n_{A_2} n_{B_2} n_{C}} \Big\| _{n_{A_1} n_{A_2}\rightarrow n_{B_1} n_{B_2}} \!\!\! \\
    &\leq  \big\|  h^{(1)}_{n_{A_1} n_{B_1} n_{C}} \big\| _{n_{A_1} \rightarrow n_{B_1} n_{C}} \big\| h^{(2)}_{n_{A_2}n_{B_2} n_{C}} \big\| _{n_{A_2} n_{C} \rightarrow n_{B_2}}.
\end{align*}
\end{lemma}

\begin{proof}
    Let $z:= z_{n_{A_1} n_{A_2}} \in (\mathbb{F}^d)^{A_1 \cup A_2}$ be arbitrary. By first applying the tensor estimate for $h^{(2)}$ from Definition~\ref{def: Tensor Norm}, we obtain
    \begin{align*}
         & \sum_{n_{B_1},n_{B_2}} \Bigg| \sum_{n_{A_1}, n_{A_2}, n_C} h^{(1)}_{n_{A_1} n_{B_1} n_{C}} h^{(2)}_{n_{A_2} n_{B_2} n_{C}} z_{n_{A_1}n_{A_2}}\Bigg| ^2 \\
          &= \sum_{n_{B_1}} \sum_{n_{B_2}} \Bigg| \sum_{n_{A_2}, n_C} h^{(2)}_{n_{A_2} n_{B_2} n_{C}} \bigg(\sum_{n_{A_1}} h^{(1)}_{n_{A_1} n_{B_1} n_{C}} z_{n_{A_1}n_{A_2}} \bigg) \Bigg| ^2 \\
          &\leq    \Big\| h^{(2)}_{n_{A_2} n_{B_2} n_{C}} \Big\|^2_{n_{A_2} n_C \rightarrow n_{B_2}} \bigg(\sum_{n_{A_2}, n_{B_1}, n_C} \Big| \sum_{n_{A_1}} h^{(1)}_{n_{A_1} n_{B_1} n_{C}} z_{n_{A_1}n_{A_2}} \Big|^2 \bigg).
    \end{align*}

    Now, using the tensor bound on $h^{(1)}$, it follows that 
    \begin{align*}
         &\sum _{n_{A_2}} \bigg( \sum_{n_{B_1}, n_C} \bigg| \sum_{n_{A_1}} h^{(1)}_{n_{A_1} n_{B_1} n_{C}} z_{n_{A_1}n_{A_2}} \bigg|^2 \bigg) \\
         \leq  &\Big\|h^{(1)}_{n_{A_1} n_{B_1} n_{C}}  \Big\|^2_{n_{A_1} \rightarrow n_{B_1}n_{C}} \bigg(\sum _{n_{A_2}} \sum _{n_{A_1}} |z_ {n_{A_1}n_{A_2}}|^2 \bigg).
    \end{align*}
By taking the supremum over all $z_{n_{A_1} n_{A_2}}$ such that 
$\sum_{n_{A_1},  n_{A_2}} \left| z_{n_{A_1} n_{A_2}} \right|^2  = 1$ and the square root of both sides, we end up with the desired bound.
\end{proof}

\subsection{Probabilistic  background} \label{subsection: Probabilistic Background}

The non-commutative Khintchine inequality is the key probabilistic tool in our proof of the Gaussian case (see Lemma~\ref{lemma: GaussianCase}). This inequality is broadly useful in random matrix theory because it allows us to upper-bound the expectation of the operator norm of a random matrix by the underlying deterministic covariance structure of the matrix, which is a computable quantity. 

\begin{lemma} [Non-commutative Khintchine inequality] \label{NoncommutativeKhintchine}
Let $X\in \mathbb{F}^{N\times N}$ be a matrix where the entries are jointly Gaussian and centered. Then $X$ can be written as follows: 
\begin{equation*}
    X = \sum_{k=1}^s g_nA_n,
\end{equation*}
where $(g_1, g_2, ..., g_n)$ are independent standard Gaussians and $A_1, A_2, ..., A_s$ are deterministic coefficient matrices.

Then, for all $p \geq 1$ , we have that
\begin{equation}
    \xp \left[ \left\|X \right\|_{\text{op}} ^{2p} \right]^{\frac{1}{2p}}
    \leq c\sqrt{p\log(N)} \max \bigg\{ \bigg\| \bigg(\sum_{k = 1}^s A_nA_n^*\bigg)  \bigg\|_{\text{op}}^{\frac{1}{2}} , \bigg\| \bigg(\sum_{k = 1}^s A_n^*A_n\bigg)  \bigg\|_{\text{op}}^{\frac{1}{2}} \bigg\},
 \end{equation}
for a constant $c$.
\end{lemma}

\begin{remark}
We omit the proof here; however, an elegant proof can be found in \cite[Theorem 3.2]{vH17}. To recover the non-commutative Khintchine inequality as stated above from \cite[Theorem 3.2]{vH17} we make a few remarks. First, we may assume, without loss of generality, that the coefficient matrices $A_n$ are all symmetric, since we can reduce to this case (see \cite[Remark 2.4]{vH17}). 
For a symmetric matric $X$, the statement of the non-commutative Khintchine inequality in \cite{vH17} is 
\begin{equation} \label{eq: khintchine1}
    \xp \left[ \text{Tr}\left[ X^{2p} \right] \right]^{\frac{1}{2p}} \leq \sqrt{2p-1} \text{Tr} \left[ \Big( \sum_{n=1}^sA_n^{2} \Big)^p \right] ^{\frac{1}{2p}}.
\end{equation}

To lower bound the LHS of \eqref{eq: khintchine1}, we note that the square symmetric matrix $X$ is diagonalizable, so 
\begin{equation} \label{eq: khintchine2}
    \xp \left[ \left\| X \right\|^{2p}_{\text{op}} \right]^{\frac{1}{2p}}  
    = \xp \left[ \left\| X^{2p} \right\|_{\text{op}} \right]^{\frac{1}{2p}} 
    \leq \xp \left[ \text{Tr}\left[ X ^{2p} \right] \right]^{\frac{1}{2p}}. 
\end{equation}

To upper-bound the RHS of \eqref{eq: khintchine1}, it holds that
\begin{align*}
    \sqrt{2p\!-\!1} \ \text{Tr}\!\left[ \left( \sum_{k=1}^s A_n^2 \right)^{\! p \ } \right]^{\frac{1}{2p}} \!\!\!
    &\leq \sqrt{2p\!-\!1}  \ \left[ N \cdot  \left\| \sum_{k=1}^s A_n^2 \right\|_{\text{op}}^{p \ } \right]^{\frac{1}{2p}}
    \\ &= \sqrt{2p\!-\!1} \ N^{\frac{1}{2p}} \max\left\{ \! \left\| \sum_{k=1}^s A_nA_n^* \right\|_{\text{op}} ^{\frac{1}{2}}\!\!\!, \left\| \sum_{k=1}^s A_n^*A_n \right\|_{\text{op}} ^{\frac{1}{2}} \right\},
\end{align*}
where the last inequality follows because every $A_n$ is symmetric and $A_n = A_n^*$. It just remains to show that, for all $p$ and $N$, the terms outside the maximum are bounded by $c\sqrt{p\log(N)}$ to satisfy \eqref{eq: khintchine2}. If $2p \geq \log(N)$, then $N^{\frac{1}{2p}} \leq N^{\frac{1}{\log(N)}}$ is bounded by a constant. Hence, we obtain the required bound.
If $2p < \log(N)$, by H\"older's inequality it follows that
\begin{equation*}
    \xp \left[ \|X\|_{\text{op}}^p \right]^{\frac{1}{p}} \leq \xp \left[ \| X\|_{\text{op}} ^{\log(N)}\right]^{\frac{1}{\log(N)}}.
\end{equation*}
Then, by applying the non-commutative Khintchine inequality to the RHS, we end up with $\sqrt{\log(N)} \cdot N^{\frac{1}{\log(N)}} \geq c\sqrt{\log(N)}$, as required.
\end{remark}

\section{Laguerre Polynomials} \label{section: Laguerre}
Recall that objective is to define the renormalization $:\prod_{j\in J} g_{n_j}^{\iota_j}:$ of the product $\prod_{j\in J} g_{n_j}^{\iota_j}$ that can detect and account for pairings. In the complex setting, it turns out to be useful to define the renormalization in terms of a Laguerre-type polynomial. For expository purposes, we also provide the construction in the real-valued case using Hermite polynomials at the end of this section. 
We need the renormalization to obtain a decoupling inequality, which is a known result for Hermite polynomials (see, for example, \cite[Section 2]{AG93}). 
We first introduce the notation. 

\begin{definition} \label{def: Laguerre notation}
Let $(g_n)_{n\in \ints^d}$ be a sequence of independent, standard, complex-valued Gaussians indexed by vectors in $\ints^d$. Let $J := \{1, 2, ..., k\}$ be the index set of $n_J = (n_1, n_2, ..., n_k) \in (\ints^d)^k$ and $\iota_J = (\iota_1, \iota_2, ..., \iota_k) \in \{-1, 1 \}^k$. For simplicity, we write $g_{n_J}^{\iota_J} : = (g_{n_1}^{\iota_1}, g_{n_2}^{\iota_2}, ..., g_{n_k}^{\iota_k})$. Since $n_J$ is an arbitrary vector in $(\ints^d)^k$, the entries are not necessarily distinct. In the inductive step, we need to separate one of the indices and so, for any $j \in J$, we define $n_{J\backslash j} := (n_1, ..., n_{j-1}, n_{j+1}, ..., n_k) \in (\ints^d)^{k-1}$ and, analogously, for $\iota_{J\backslash j}$.  
\end{definition}

Fix $n_J \in (\ints^d)^k$. Then, we define the two functions $\sigma_{n_J} :\ints^d\rightarrow \ints_{\geq 0}$ and $\mu_{n_J} : \ints^d\rightarrow \ints$, by
\begin{equation*}
    \sigma_{n_J}(n) = \left| \{j \in J \mid n_j = n\}\right|\quad \text{ and }\quad \mu_{\iota_J}(n) = \sum_{j\in J, n_j = n} \iota_j
\end{equation*}
where $n_j$ and $\iota_j$ are the $j$th coordinates of the vectors $n_J$ and $\iota_J$, respectively. For each $n$, the function $\sigma$ counts the number of times $g_n^{\pm 1}$ appears in the product. The function $\mu$ is the signed difference between the number of $g_n$'s and $\overline{g_n}$'s in the product. Lastly, we let $\sgn (\mu_{n_J}(n)) \in \{-1, 1 \} $ denote the sign of $\mu$. We will consider the $\mu =0$ case separately.

\begin{definition} [Standard Laguerre Polynomials]
Since we will require the generalized (associated) Laguerre polynomials, we provide the recursive construction. For every $\alpha \geq 0$, the first two Laguerre polynomials are
\begin{align*}
    L_0^{\alpha}(x) = 1 \quad \text{and} \quad  L_1^{\alpha}(x) = 1 + \alpha -x.
\end{align*}
For $k \geq 2$, the Laguerre polynomials are given by the recurrence relation
\begin{align*}
    L_{k+1}^{\alpha}(x) = \frac{(2k+1+\alpha-x) L_k^{\alpha} - (k+\alpha) L_{k-1}^{\alpha}(x)}{k+1}.
\end{align*}

We often care about the simple Laguerre polynomials, which are the $\alpha =0$ case. For the reader's reference, we provide the first few simple Laguerre polynomials
\begin{align*}
    &L_0(x) = 1,  \qquad \quad L_1(x) = -x + 1, \qquad \quad L_2(x) = \tfrac{1}{2} (x^2 - 4x +2), \\
    &L_3(x) = \tfrac{1}{3!} (-x^3 + 9x^2 - 18x + 6),  \qquad \quad L_4(x) = \tfrac{1}{4!} (x^4 - 16x^3 + 72x^2 - 96x +24).
\end{align*}
\end{definition}

For further reference on Laguerre polynomials, see \cite[Section 13.2]{MR05}. We now construct our Laguerre-type polynomials, which are inspired by the use of Laguerre polynomials in the work of Oh-Thomann \cite{OT18}. 
\begin{definition} [Laguerre-type polynomials] \label{def: Laguerre type polynomials}
    Suppose $\sigma \in \ints_{>0}$ and $\mu \in \ints_{\geq 0}$, where $\sigma \geq \mu$, then we define
\begin{equation*}
    \mathcal{L}(\sigma, \mu, g_n)  = (-1)^{\frac{\sigma-|\mu|}{2}} \left(\tfrac{\sigma-|\mu|}{2} \right)!\ L_{\tfrac{\sigma-|\mu|}{2}}^{|\mu|}(|g_n|^2) g_n^{\mu}.
\end{equation*}
In what follows, $\frac{\sigma-|\mu|}{2}$ will be an integer.
Let $(g_n)_{n\in \ints ^d}$ be as given in Definition~\ref{def: Laguerre notation}. Fix some $n_J = (n_1, \dots, n_k) \in (\ints^d)^k$ and $\iota_J = (\iota_1, \dots, \iota_k) \in \{1, -1\}^k$, then define 
\begin{equation*}
     \mathcal{L}(g_{n_J}^{\iota_J}) := \prod_{n\in \ints^d} \mathcal{L}(\sigma_{n_J}(n), \mu_{n_J}(n), g_n).
\end{equation*}

Although it appears that we are handling an infinite product, all but finitely many terms are equal to $1$. Indeed, for all $n \in \ints^d$ where $g_n$ does not appear in the product and, hence, $\sigma_{n_J}(n) =0$, the polynomial $\mathcal{L}(\sigma_{n_J}(n),\mu_{n_J}(n), g_n) = 1$. 
\end{definition}

We first discuss a few examples, then we will show that this polynomial satisfies the required properties.

\begin{example} \label{ex: Laguerre Example 1}
 Let $g_n$ be a standard complex-valued Gaussian. Consider the product $g_n^{\alpha}\overline{g_n}^{\beta}$, where $\alpha \geq \beta$.
The renormalization is given by 
\begin{equation*} 
    \mathcal{L}(\alpha + \beta, \alpha - \beta, g_n) =(-1)^\beta \beta! L^{\alpha -\beta}_{\beta}(|g_n|^2) g_n^{(\alpha -\beta)}.
\end{equation*}
Let $\sigma$ and $\mu$ be as in Definition~\ref{def: Laguerre type polynomials}. We can check that, when $\alpha \geq \beta$, $\frac{\sigma + \mu}{2}$ equals the number of times $g_n$ appears in the product and $\frac{\sigma - \mu}{2}$ is the number of times $\overline{g_n}$ appears in the product. We will use this observation in Lemma~\ref{lemma: Laguerre-type partial derivatives}. 
\end{example}

\begin{example}
    We list the three possible polynomials for the normalization of the product $g_n^{\iota_1}g_n^{\iota_2}g_n^{\iota_3}g_n^{\iota_4}$, depending on the set of signs $\{\iota_1, \iota_2, \iota_3, \iota_4 \} \in \{-1,+1\}^4$.
    \begin{enumerate}
        \item $g_n^4$ is normalized to $\el(4,4,g_n) = g_n^4$.
        \item $g_n^3\overline{g_n}$ is normalized to $\el(4,2,g_n) = g_n^3\overline{g_n} - 3g_n^2$.
        \item $g_n^2\overline{g_n}^2$ is normalized to $\el(4,0,g_n) = g_n^2\overline{g_n}^2 - 4g_n\overline{g_n}+2$.
    \end{enumerate}
    Since Gaussians are rotation invariant, $g_n^3\overline{g_n}$ and $g_n\overline{g_n}^3$ have the same normalization (up to conjugation). We make a few remarks that made evident here. First, the polynomials indeed have a monomial equal to the product $g_n^{\iota_1}g_n^{\iota_2}g_n^{\iota_3}g_n^{\iota_4}$ that we are renormalizes. Second, $\sigma$ denotes the degree of the polynomial $\mathcal{L}(\sigma, \mu, g_n)$ and this is precisely equal to the number of Gaussians in the product. 
\end{example}

In the following lemma we will check that $ \mathcal{L}(g_{n_J}^{\iota_J})$ is indeed the renormalization.

\begin{lemma} [Expectation of the Laguerre-type polynomials] \label{lem: Laguerre-type expectation}
In the present setting, for all $n_J \in (\ints^d)^k$ and $\iota_J \in \{1, -1\}^k$, it holds that 
    \begin{equation*}
        \mathbb{E} \left[\mathcal{L}(g_{n_J}^{\iota_J})\right] = 0.
    \end{equation*}
\end{lemma}
\begin{proof}
    For distinct $n\in \ints^d$, the Gaussians $g_n$ are independent. Hence, 
    \begin{equation*}
    \xp \big[\mathcal{L}(g_{n_J}^{\iota_J} ) \big] = \xp \bigg[ \ \prod_{n\in \ints^d} \mathcal{L}(\sigma_{n_J}(n), \mu_{n_J}(n), g_n) \bigg]
    =  \prod_{n \in \ints^d} \xp \left[ \mathcal{L}(\sigma_{n_J}(n), \mu_{n_J}(n), g_n) \right].
\end{equation*}
   Suppose first that $\mu_{n_J}(n)\neq 0$. Referring back to the construction in Definition~\ref{def: Laguerre type polynomials}, observe that $\mathcal{L}(\sigma_{n_J}(n), \mu_{n_J}(n), g_n)$ is composed of a polynomial in terms of $|g_n|^2$ and $g_n^{\mu}$. Therefore, every monomial in $\mathcal{L}(\sigma_{n_J}(n), \mu_{n_J}(n), g_n)$ has an asymmetric number of $g_n$'s and $\overline{g_n}$'s and consequently has expectation zero, as required. Now suppose that $\mu_{n_J}(n) = 0$, then
    $$\mathcal{L}(\sigma_{n_J}(n), \mu_{n_J}(n), g_n) = (-1)^{\frac{\sigma_{n_J}(n)}{2}} \left(\tfrac{\sigma_{n_J}(n)}{2} \right) L_{\tfrac{\sigma_{n_J}(n)}{2}}(|g_n|^2).$$
    The random variable $|g_n|^2$ has probability density $e^{-x}$ for $x \geq 0$. By the orthogonality of the Laguerre polynomials with respect to the inner product $\int_0^\infty f(x)g(x)e^{-x}dx$ (see \cite[Section 13.2]{MR05} for details), for any $\sigma \geq 1$, it follows that 
    \begin{align*}
        \xp[L_{\sigma}(|g_n|^2)] &=  \xp[L_{\sigma}(|g_n|^2) \cdot L_0(|g_n|^2)] 
        = \int L_{\sigma}(x) \cdot L_0(x) e^{-x} dx
        =0,
    \end{align*}
    concluding the proof.
\end{proof}

In probabilistic decoupling, we will need to interpolate between two independent families of Gaussians. The following lemma will allow us to set up the inductive argument, using independence of $(\partial_\varphi (g_n(\varphi)))_{n\in \ints^d}$ and $(g_n(\varphi))_{n\in \ints^d}$. For a more precise explanation see Lemma~\ref{lemma: Probabilistic Decoupling}.

\begin{lemma}[Partial derivatives of the Laguerre-type polynomials] \label{lemma: Laguerre-type partial derivatives}
    With the same assumptions as above and $k \geq 1$, let $(\tilde{g}_n)_{n\in \ints^d}$ be an independent copy of $(g_n)_{n\in \ints^d}$. Then for any $\varphi \in [0, \frac{\pi}{2}]$, define $g_n(\varphi):= sin(\varphi)g_n + cos(\varphi)\tilde{g}_n$. We use $g_{n_J}^{\iota_J}(\varphi)$ to denote $(g_{n_1}^{\iota_1}(\varphi), ..., g_{n_k}^{\iota_k}(\varphi))$. Then 
    \begin{align*}
        & \partial_\varphi \left[ \mathcal{L}(g_{n_J}^{\iota_J} (\varphi)) \right]  \ = \sum_{j\in J} \partial_\varphi \big[ g_{n_j}^{\iota_j}(\varphi) \big]\mathcal{L}(g_{n_{J \setminus j}}^{\iota_{J \setminus j}}(\varphi)).
    \end{align*}
\end{lemma}

\begin{proof} 
We will require the following two properties about the derivatives of generalized Laguerre polynomials:
\begin{equation} \label{eq: Laguerre derivative property (1)}
        \frac{d}{dx}L_m^{(\alpha)}(x) = (-1) L_{m-1}^{\alpha+1}(x)
\end{equation}
and 
\begin{equation} \label{eq: Laguerre derivative property (2)}
    \frac{d}{dx}x^{\alpha}L_m^{(\alpha)}(x) = (m+\alpha) x^{\alpha -1} L_m^{(\alpha-1)}(x) .
\end{equation}
For further reading on the properties of Laguerre polynomials see \cite[Section 13.2]{MR05}.
By repeatedly applying product rule and chain rule, we obtain that 
\begin{align}
      & \nonumber \partial_\varphi  \mathcal{L}(g_{n_J}^{\iota_J}(\varphi))=  \partial_\varphi \prod_{n \in \ints^d}\mathcal{L}(\sigma_{n_J}(n), \mu_{n_J}(n), g_n(\varphi)) \\
    &=  \nonumber \sum_{m\in \ints^d} \partial_\varphi[\mathcal{L}(\sigma_{n_J}(m), \mu_{n_J}(m), g_m(\varphi))] \prod_{n \in \ints^d, n \neq m}\mathcal{L}(\sigma_{n_J}(n), \mu_{n_J}(n), g_n(\varphi)).
\end{align}
 The polynomial $\mathcal{L} (\sigma_{n_J}(m), \mu_{n_J}(m), g_m(\varphi)) = 1$ when $\sigma_{n_J} (m) = 0$ and vanishes under taking the partial derivative. Therefore, the only nonzero terms of the sum are those where $m = n_j$ for some $j\in J$.

For simplicity, we denote $\sgn(\mu_{n_J}(m))$ by $\sgn(\mu)$. We first evaluate the partial derivative with respect to $g_m^{\sgn(\mu)}$. We may assume there exists at least one index $j \in J$ such that $n_j = m$ and $\iota_j = \sgn(\mu)$. By expanding $g_m^{\sgn(\mu)} = |g_m|^2 / g_m^{-\sgn(\mu)}$ and viewing $g_m^{-\sgn(\mu)}$ as constant, we may directly apply \eqref{eq: Laguerre derivative property (2)} as follows
\begin{align*}
    & \partial_{g_m^{\sgn(\mu)}} \bigg[  (-1)^{\frac{\sigma-|\mu|}{2}} \Big(\tfrac{\sigma-|\mu|}{2} \Big)!\  L_{\frac{\sigma-|\mu|}{2}}^{|\mu|}(|g_m|^2) \  g_m^{\mu} \bigg] \\
    &=  \tfrac{\sigma+|\mu|}{2} (-1)^{\frac{\sigma-|\mu|}{2}} \Big(\tfrac{\sigma-|\mu|}{2}\Big)!\   L_{\frac{(\sigma-1)-(|\mu|-1)}{2}}^{|\mu|-1}(|g_m|^2) \  g_m^{(\mu-\sgn(\mu))} \\ 
    &= \tfrac{\sigma+|\mu|}{2}\  \mathcal{L}(\sigma_{n_J}(m)-1, \mu_{n_J}(m)-\sgn(\mu), g_m(\varphi)) \\
    &= \tfrac{\sigma+|\mu|}{2}\   \mathcal{L}(\sigma_{n_{J\setminus j}}(m), \mu_{n_{J\setminus j}}(m), g_m(\varphi)).
\end{align*}
Recall from Example~\ref{ex: Laguerre Example 1} that $\tfrac{\sigma+|\mu|}{2}$ is precisely equal to the number of distinct $j \in J$, such that $n_j = m$ and $\iota_j = \sgn(\mu)$. Equivalently, this is the number of times $g_m^{\sgn(\mu)}$ appears in the product.

We now evaluate the partial derivative with respect to $g_m^{-\sgn(\mu)}$. Just as above, we may assume there is an index $j' \in J$ such that $n_{j'} = m$ and $\iota_{j'} = -\sgn(\mu)$. Viewing $g_m^{\sgn(\mu)}$ as a constant, we apply chain rule and \eqref{eq: Laguerre derivative property (1)} as follows
\begin{align*}
    &\partial_{g_m^{-\sgn(\mu)}} \bigg[ (-1)^{\frac{\sigma-|\mu|}{2}} \Big(\tfrac{\sigma-|\mu|}{2} \Big)!\  L_{\frac{\sigma-|\mu|}{2}}^{|\mu|}(|g_m|^2) \  g_m^{\mu} \bigg] \\
    &=  \tfrac{\sigma-|\mu|}{2} (-1)^{\frac{\sigma-|\mu|-1}{2}} \Big(\tfrac{\sigma-|\mu|-1}{2}\Big)!\   L_{\frac{(\sigma-1)-(|\mu|+1)}{2}}^{|\mu|+1}(|g_m|^2) \ g_m^{(\mu+\sgn(\mu))}  \\ 
    &= \tfrac{\sigma-|\mu|}{2} \ \mathcal{L}(\sigma_{n_{J\setminus {j'}}}(m), \mu_{n_{J\setminus {j'}}}(m), g_m(\varphi)).
\end{align*}
Observe that $\tfrac{\sigma-|\mu|}{2}$ is precisely the number of times that $g_m^{-\sgn(\mu)}$ appears in the product. There is one technicality concerning $\mu_{n_J}(m)=0$. One can check that by direct application of \eqref{eq: Laguerre derivative property (1)} and chain rule (as done above) that we obtain the desired equation.
Moreover, whenever $n \neq n_j$, it follows that $$\sigma_{n_J}(n) = \sigma_{n_{J\setminus j}}(n) \quad \text{and} \quad  \mu_{\iota_J}(n) = \mu_{\iota_{J\setminus j}}(n).$$
Combining the above, it holds that
\begin{align*}
    &\sum_{m\in \ints^d} \partial_\varphi\big[\mathcal{L}(\sigma_{n_J}(m), \mu_{n_J}(m), g_m(\varphi)) \big] \prod_{n \in \ints^d, n \neq m}\mathcal{L}(\sigma_{n_J}(n), \mu_{n_J}(n), g_n(\varphi)) \\
    &= \sum_{j \in J} \partial_\varphi(g_{n_j}^{\iota_j})\mathcal{L}(\sigma_{n_{J \setminus j}}({n_j}), \mu_{n_{J \setminus j}}({n_j}), g_{n_j}(\varphi)) \prod_{n \in \ints^d, n \neq n_j} \!\!\! \mathcal{L}(\sigma_{n_{J \setminus j}}(n), \mu_{n_{J \setminus j}}(n), g_n(\varphi)).
\end{align*}
The desired bound follows directly by recalling the definition of $\mathcal{L}(g_{n_{J \setminus j}}^{\iota_{J \setminus j}}(\varphi))$. 

\end{proof}

We conclude this section by providing the construction for the polynomial that renormalizes the product of real-valued Gaussians. We will only state and prove the abstract random tensor estimate for complex-valued Gaussians, but the analogous statement for real-valued Gaussians follows by an almost identical argument.
\begin{definition} [\textbf{Hermite Polynomials}]\label{def: Hermite Polynomials}
    Let $H_n(x)$ denote the $n$th single variable Hermite polynomial. Set $H_0 = 1$. We recursively construct $H_n$, for $n\geq 1$ by requiring that
\begin{enumerate} 
    \item [(i)] $\frac{d}{dx}H_n(x) = nH_{n-1}(x)$ and
    \item [(ii)] $\xp \left[ H_n(X) \right] = 0 $, when $X$ is a standard real-valued Gaussian.
\end{enumerate}

This uniquely defines the $n$th Hermite polynomial, because (i) determines the terms of degree $\geq 1$ and (ii) determines the constant terms.
We can use this to compute the first few single variable Hermite polynomials:
\begin{align*}
    & H_1(x) = x, \quad \ \  H_2(x) = x^2 -1, \quad \ \  H_3(x) = x^3 - 3x, \quad \ \   H_4(x) = x^4 - 6x^2  + 3.
\end{align*}
For reference on the Hermite polynomials, see \cite[Section 2]{H21}. 
\end{definition}

\begin{definition} [\textbf{Hermite-type Polynomials}] 
    With the same notation as above. Let $(g_n)_{n \in \ints^d}$ be a sequence of pairwise independent, real-valued Gaussians and fix some $n_J = (n_1, n_2, ..., n_k) \in (\ints^d)^k$. Then, the  renormalization $:\prod_{j\in J} g_{n_j}:$ of the product $\prod_{j\in J} g_{n_j}$ is given by 
    \begin{align*}
     H(g_{n_J}) :=  \prod_{n\in \ints^d} H_{\sigma_{n_J}(n)}(g_n).
\end{align*}
\end{definition}

\begin{lemma} [Properties of the Hermite-type Polynomials]
Moreover, the Hermite-type polynomials satisfy the following two properties:
\begin{enumerate}
    \item [(i)] $\xp \left[ H(g_{n_J}) \right] = 0$,
    \item [(ii)] With $g(\varphi)$ as defined in Lemma~\ref{lemma: Laguerre-type partial derivatives}, the partial derivative with respect to $\varphi$ is given by $$\partial_\varphi H(g_{n_J}) = \sum_{j\in J} \partial(g_{n_j})H(g_{n_{J \setminus j}}(\varphi)).$$
\end{enumerate}
\end{lemma} 
\begin{remark}
While we omit a full proof, we provide a sketch below.
The proof of (i) follows from the independence assumption on the sequence $(g_n)_{z\in \ints ^d}$ and property (ii) in Definition~\ref{def: Hermite Polynomials}. The proof of (ii) follows by using chain rule and product rule as in the proof of Lemma~\ref{lemma: Laguerre-type partial derivatives}. Then, by using property (i) in Definition~\ref{def: Hermite Polynomials}, if $m = n_j$ for some $j\in J$, we observe that 
$$\partial_{g_{m}}H_{\sigma_{n_J}(m)}(g_{m}) = \sigma_{n_J}(m) \cdot H_{\sigma_{n_J}(m)-1}(g_{m}) =\sigma_{n_J}(m) \cdot H_{\sigma_{n_{J \setminus j}}(m)}(g_{m}),$$
where $\sigma_{n_J}(m)$ is precisely the number of times $g_m$ appears in the product.
\end{remark}

\section{The Proof of the Abstract Random Tensor Estimate}
We begin by recalling the statement that we will prove. We omit the proof of the real-valued case since it follows by an analogous argument to the complex-valued case.

Under the same assumptions as Theorem~\ref{Abstract Random Tensor Estimate}, let $G = G_{n_An_B}$ be the random tensor defined as
\begin{equation*}
        G_{n_An_B} := \sum_{n_J \in (\ints^d)^k} h_{n_Jn_An_B} \mathcal{L}(g_{n_J}^{\iota_J}).
\end{equation*}
Then it holds for all $p \geq 1$ that 
    \begin{equation*}
       \xp\big[\|G\|^p_{n_A\rightarrow n_B}\big]^{\frac{1}{p}} \leq C p^{\frac{k}{2}} (\log N)^{\frac{k}{2}} \max_{X \cupdot Y = {J}} \|h\|_{n_An_X \rightarrow n_Bn_Y},
    \end{equation*}
where $C=C(k,d,|A|,|B|)$ is a constant.

The proof will proceed by induction on the number of Gaussians in the product, which is equivalent to the size of $J$ or the degree of the renormalizing polynomial. We use two key lemmas: the \textit{Gaussian case} is the base case and \textit{probabilistic decoupling} is the main tool in the inductive step. The key new tools in our proof are (1) the direct use of the non-commutative Khintchine inequality in the Gaussian case and (2) the use of Laguerre-type polynomials in the probabilistic decoupling to account for pairings, allowing us to remove the the square-free requirement. 
Note that if we replace the Laguerre-type polynomials in this proof with Hermite polynomials, the real-valued case follows by an almost identical argument. 

\begin{lemma} [Gaussian case] \label{lemma: GaussianCase}
Let $h=h_{{n_0}{n_A}{n_B}}$ and $(g_n)_{n\in \ints^d}$ be as given in Theorem~\ref{Abstract Random Tensor Estimate}. Then, for all $p\geq 1$,
\begin{align*} \label{eq:GaussianCase}
    \nonumber & \xp \bigg[ \Big\| \sum_{n_0} h_{n_0n_An_B}g_{n_0} \Big\|^p_{n_A \rightarrow n_B} \bigg]^{\frac{1}{p}} \\
    &\lesssim c\sqrt{p\log N} \max \left\{ \left\| h_{n_0n_An_B} \right\|_{n_0n_A \rightarrow n_B}, \left\| h_{n_0n_An_B} \right\|_{n_A \rightarrow n_0n_B}\right\},
\end{align*} 
where $c$ is a constant in terms of $|A|$, $|B|$, and $d$.
\end{lemma}

\begin{proof}
    There are two important ingredients in this proof. On the probabilistic side, we use the non-commutative Khintchine inequality (see Lemma~\ref{NoncommutativeKhintchine}), to bound the expectation of the operator norm of the random tensor by the operator norm of some deterministic tensor. Then, on the deterministic side, we use the merging estimate (see Lemma~\ref{lemma: MergingEstimate}) to bound the operator norm of this deterministic tensor in terms of tensor norm of $h$.
    \newline

    For each $n_0 \in \ints^d$, we let $\mathcal{T}_{n_0} : \ell^2_{n_A} \rightarrow \ell^2_{n_B}$ be the linear operator:
    \begin{equation*}
        (\mathcal{T}_{n_0}z)_{n_B} = \sum_{n_A}h_{n_0n_An_B}z_{n_A}.
    \end{equation*}
    Then, using the definition of $\mathcal{T}_{n_0}$, we can write 
    \begin{align} \label{eq: gc1}
        \bigg\| \sum_{n_0} h_{n_0n_An_B}g_{n_0} \bigg\|^2_{n_A \rightarrow n_B}  
        &= \bigg\| \sum_{n_0}  g_{n_0} \mathcal{T}_{n_0} \bigg\|^2_{\text{op}},
    \end{align} 
    where $\| \cdot \|_{\text{op}}$ is the usual operator norm.
   The equality follows from the observation that, when the tensor is viewed as a linear operator, the tensor norm corresponds with the usual operator norm.

  Taking the expectation of \eqref{eq: gc1}, we can use the non-commutative Khintchine inequality (Lemma~\ref{NoncommutativeKhintchine}) as follows
    \begin{align}
         \xp \bigg[  \Big\| \sum_{n_0} \mathcal{T}_{n_0}\ g_{n_0} \Big\|_{\text{op}}^p\bigg]^{\frac{1}{p}} 
        \!\! \lesssim &\  c\sqrt{\!p \log N}  \max \bigg\{ \! \Big\|\! \sum_{n_0} \mathcal{T}_{n_0}  \mathcal{T}_{n_0}^*  \Big\|^{\frac{1}{2}}_{\text{op}} , \Big\| \!\sum_{n_0} \mathcal{T}_{n_0}^*  \mathcal{T}_{n_0}  \Big\|^{\frac{1}{2}}_{\text{op}} \!\bigg\}, \label{eg: gc2}
    \end{align}
    where $c$ is a constant in terms of the dimension of the spaces to and from which we are mapping. This dimension is at most $D N^D$ for a for a constant $D$ depending on $|A|$, $|B|$, and $d$.

    We estimate the two arguments in \eqref{eg: gc2} separately. Observe that $(\sum_{n_0} \mathcal{T}_{n_0}^* \mathcal{T}_{n_0})$ is a linear operator mapping $\ell^2_{n_A} \rightarrow \ell^2_{n_A'}$ and, similarly, $(\sum_{n_0} \mathcal{T}_{n_0} \mathcal{T}_{n_0}^*)$ is a linear operator mapping $\ell^2_{n_B'} \rightarrow \ell^2_{n_B}$. Using the definition of $\mathcal{T}_{n_0}$, we can express the first argument in the maximum as 
    \begin{equation*}
        \Big( \sum_{n_0} \mathcal{T}_{n_0}^* \mathcal{T}_{n_0} \Big)_{n'_A n_A} = \sum_{n_0, n_B} \overline{h_{n_0n_A'n_B}}h_{n_0n_An_B}.
    \end{equation*}

    Then, using the merging estimate (Lemma~\ref{lemma: MergingEstimate}), we have that 
    \begin{align*}
        \bigg\| \sum_{n_0} \mathcal{T}_{n_0}^* \mathcal{T}_{n_0} \bigg\|_{\text{op}} &=  \bigg\| \sum_{n_0, n_B} \overline{h_{n_0n_A'n_B}}h_{n_0n_An_B} \bigg\|_{n_A\rightarrow n'_A} \\ 
        &\leq  \big\| h_{n_0n_An_B} \big\|_{n_A\rightarrow n_0n_B} \big\| \overline{h_{n_0n'_An_B}} \big\|_{n_0n_B\rightarrow n'_A}    \\ 
        & = \big\| h_{n_0n_An_B} \big\|_{n_A\rightarrow n_0n_B}^2,
    \end{align*}
    where the last equality follows from duality of the tensor norm in \eqref{eq: Tensor Norm Duality}. Analogously,
    \begin{equation*}
          \left\| \sum_{n_0} \mathcal{T}_{n_0} \mathcal{T}_{n_0}^* \right\|_{\text{op}} \leq \left\| h_{n_0n_An_B} \right\|_{n_0n_A\rightarrow n_B}^2.
    \end{equation*}

    Therefore, we can combine these bounds with \eqref{eg: gc2}, to get the desired inequality.
\end{proof}

\begin{lemma}[Probabilistic decoupling]\label{lemma: Probabilistic Decoupling}
Let $h_{n_Jn_An_B}$ and $(g_n)_{n\in \ints^d}$ be as in Theorem~\ref{Abstract Random Tensor Estimate} and let $(\tilde{g}_n)_{n\in \ints^d}$ be an independent copy of $(g_n)_{n\in \ints^d}$. For all $p \geq 1$, it holds that 
\begin{align*}
    & \xp \bigg[ \Big\| \sum_{n_J \in (\ints^d)^k} h_{n_Jn_An_B} \mathcal{L}(g_{n_J}^{\iota_J})\Big\| _{n_A \rightarrow n_B}^p  \bigg]^{\frac{1}{p}} 
    \\
    &\leq \frac{\pi}{2} \sum_{j \in J}      \xp \bigg[ \Big\| \sum_{n_J\in (\ints^d)^k} h_{n_Jn_An_B}  \tilde{g}_{n_j} \mathcal{L}(g_{n_{J \setminus j}}^{\iota_{J \setminus j}}) \Big\| _{n_A \rightarrow n_B}^p  \bigg] ^{\frac{1}{p}}.
\end{align*}
\end{lemma}

\begin{proof}
We prove this by induction on the size of $J$. When $J = \{1\}$, we have that $\mathcal{L}(g_{n_J}^{\iota_J}) =g_{n_1}$ and $\mathcal{L}(g_{n_{J\setminus j}}^{\iota_{J\setminus j}}) = 1$. The inequality follows immediately. \newline 
Now, consider $k \geq 2$ and $J = \{1, 2, ..., k\}$. For expository purposes, we write $\xp_g$ and $\xp_{\tilde{g}}$ to denote the expectations taken over $g$ and $\tilde{g}$, respectively. Furthermore, for any $a = (a_{n_J})_{n_J\in (\ints^d)^k}$, we introduce the notation
\begin{equation} \label{eq: pd1}
    F(a_{n_J}) := \bigg\| \sum_{n_J\in (\ints^d)^k} h_{n_Jn_An_B} a_{n_J} \bigg\| _{n_A \rightarrow n_B},
\end{equation}
which is convex and one-homogeneous.

By Lemma~\ref{lem: Laguerre-type expectation}, we have that $\xp [\mathcal{L}(\tilde{g}_{n_J}^{\iota_J})] = 0$ when $k \geq 1$. Consequently, it holds that
\begin{align*} 
     \xp_g \bigg[ \Big\| \sum_{n_J\in (\ints^d)^k} h_{n_Jn_An_B} \mathcal{L}(g_{n_{J}}^{\iota_{J }}) \Big\|_{n_A \rightarrow n_B}^p \bigg] ^{\frac{1}{p}} & =\xp_g \left[ F (\mathcal{L}(g_{n_{J}}^{\iota_{J}}))^p  \right]^{\frac{1}{p}} \\ 
      & = \xp_g \left[ F(\mathcal{L}(g_{n_{J}}^{\iota_{J}}) - \xp_{\tilde{g}}[\mathcal{L}(\tilde{g}_{n_{J}}^{\iota_{J}})])^p \right] ^{\frac{1}{p}} \\
     &\leq \  \xp_g \xp_{\tilde{g}}\left[ F(\mathcal{L}(g_{n_{J }}^{\iota_{J}})- \mathcal{L}(\tilde{g}_{n_{J}}^{\iota_{J}}))^p \right]^{\frac{1}{p}},
\end{align*}
where the last line follows by Jensen's inequality. For any $\varphi \in [0, \frac{\pi}{2}]$, we define  
\begin{equation} \label{eq:gaussianphi}
    g(\varphi) := \sin(\varphi)g + \cos(\varphi)\tilde{g}.
\end{equation}
By definition, $g(0) = \tilde{g}$ and $g(\frac{\pi}{2}) = g$, and thus $g(\varphi)$ interpolates between $g$ and $\tilde{g}$. We also observe that
\begin{equation*}
    \partial_\varphi g(\varphi) = \cos(\varphi)g - \sin(\varphi)\tilde{g}.
\end{equation*}
Since $(\sin(\varphi), \cos(\varphi))$ and $(\cos(\varphi), -\sin(\varphi))$ are orthonormal in $\reals^2$, the rotational invariance of Gaussians implies that, for all $\varphi \in [0, \frac{\pi}{2}]$, the following are equal in distribution:
\begin{equation*} 
    (g(\varphi),\partial_\varphi g(\varphi)) \overset{d}{=}     (\partial_\varphi g(\varphi), g(\varphi)) \overset{d}{=} (g, \tilde{g}).
\end{equation*}

By substituting in \eqref{eq:gaussianphi} and the triangle inequality, it follows that 
\begin{align}
    & \nonumber \ \xp_g \xp_{\tilde{g}}\left[ F \left( \mathcal{L}(g_{n_{J }}^{\iota_{J}}) - \mathcal{L}(\tilde{g}_{n_{J }}^{\iota_{J}}) \right)^p \right]^{\frac{1}{p}} \\
    =& \nonumber \  \xp_g \xp_{\tilde{g}}\left[ F\left( \mathcal{L}\big(g_{n_{J }}^{\iota_{J}}\big(\tfrac{\pi}{2}\big)\big) - \mathcal{L}\big(g_{n_{J }}^{\iota_{J}}(0)\big)\right)^p \right]^{\frac{1}{p}} \\
    =& \nonumber  \xp_g \xp_{\tilde{g}}\bigg[ F\bigg( \int_0^{\frac{\pi}{2}} \partial_{\varphi} \mathcal{L}(g_{n_{J }}^{\iota_{J}}(\varphi)) d\varphi
    \bigg)^p \bigg]^{\frac{1}{p}} \\ 
    \leq &  \int_0^{\frac{\pi}{2}} \xp_g \xp_{\tilde{g}}\bigg[ F\Big( \partial_{\varphi}   \mathcal{L}(g_{n_{J }}^{\iota_{J}}(\varphi)) \ \Big)^p \bigg]^{\frac{1}{p}} d\varphi. \label{eq: pd2}
\end{align}

By Lemma~\ref{lem: Laguerre-type expectation} and the triangle inequality, the integrands in \eqref{eq: pd2} are given by 

\begin{align*}
      \xp_g \xp_{\tilde{g}}\bigg[ F\Big( \partial_{\varphi}  \mathcal{L}(g_{n_{J }}^{\iota_{J}}(\varphi))  \Big)^p \bigg]^{\frac{1}{p}}  &= \xp_g \xp_{\tilde{g}}\bigg[ F \bigg( \sum_{j\in J}  \partial_{\varphi}\big[g_{n_j}(\varphi) \big] \mathcal{L}\big(g_{n_{J\backslash j}}^{\iota_{J\backslash j}}(\varphi) \big) \bigg)^p \bigg]  ^{\frac{1}{p}} \\
    & \leq \sum_{j \in J}  \xp_g \xp_{\tilde{g}}\bigg[ F \Big(  \partial_{\varphi}\big[ g_{n_j}(\varphi)\big] \mathcal{L} \big( g_{n_{J\backslash j}}^{\iota_{J\backslash j}}(\varphi)\big) \Big)^p \bigg]^{\frac{1}{p}}.
\end{align*}
Using that $(g(\varphi),\partial_\varphi g(\varphi))$ and $(g, \tilde{g})$ are equal in distribution and evaluating the integral, we find that 
\begin{align*}
    \eqref{eq: pd2} & =\int_0^{\frac{\pi}{2}} \sum_{j \in J} \xp_g \xp_{\tilde{g}}\Big[ F \left(  \tilde{g}_{n_j} \mathcal{L} \big( g_{n_{J\backslash j}}^{\iota_{J\backslash j}}\big) \right)^p \Big]^{\frac{1}{p}} d\varphi \\
    & = \ \frac{\pi}{2} \sum_{j \in J} \xp_g \xp_{\tilde{g}}\Big[ F \left(  \tilde{g}_{n_j}\mathcal{L} \big( g_{n_{J\backslash j}}^{\iota_{J\backslash j}} \big)\right)^p \Big]^{\frac{1}{p}}.
\end{align*}

By recalling the definition of $F$, we have our desired inequality. 
\end{proof}

\begin{proof} [Proof of Theorem~\ref{Abstract Random Tensor Estimate}]
The proof follows by induction on $|J|$.
When $J = \{1\}$, then $\mathcal{L}(g_{n_1}^{\iota _1}) = g_{n_1}^{\iota _1}$, the inequality follows directly from the Gaussian case (Lemma~\ref{lemma: GaussianCase}).
Let $k\geq 2$ and $J = \{1, 2, ..., k\}$. By probabilistic decoupling (Lemma~\ref{lemma: Probabilistic Decoupling}), it follows that 
\begin{align}
    \nonumber \xp \big[ \|G\|^p_{n_A\rightarrow n_B} \big] ^{\frac{1}{p}} 
    &=  \xp \Bigg[ \bigg\| \sum_{n_J \in (\ints^d)^k} h_{n_Jn_An_B} \mathcal{L}(g_{n_J}^{\iota_J}) \bigg\|^p_{n_A\rightarrow n_B} \Bigg] ^{\frac{1}{p}} \\ 
    &\leq \frac{\pi}{2} \sum_{j \in J} \xp \Bigg[ \bigg\| \sum_{n_J\in (\ints^d)^k} h_{n_Jn_An_B}  \tilde{g}_{n_j} \mathcal{L}(g_{n_{J\backslash j}}^{\iota_{J\backslash j}})\bigg\| _{n_A \rightarrow n_B}^p  \Bigg] ^{\frac{1}{p}}. \label{eq: Main Proof Induction}
\end{align}

We estimate each term in the sum separately. Given some $j \in J$, we define $\tilde{J} := J\backslash j$ and $n_{\tilde{J}} = (n_1, ..., n_{j-1}, n_{j+1}, ..., n_k) \in (\ints^d)^{k-1}$.
By viewing $$\sum_{n_J\in (\ints^d)^k} h_{n_Jn_An_B} \mathcal{L}(g_{n_{J\backslash j}}^{\iota_{J\backslash j}}) $$ as the tensor, we may apply the Gaussian case to $\tilde{g}_{n_j}^{\iota_j}$ as follows
\begin{align}
    \nonumber &\ \xp \bigg[ \Big\| \sum_{n_J\in (\ints^d)^k} h_{n_Jn_An_B}   \tilde{g}_{n_j}^{\iota_j} \mathcal{L}\big(g_{n_{J\backslash j}}^{\iota_{J\backslash j}}\big) \Big\| _{n_A \rightarrow n_B}^p  \bigg] ^{\frac{1}{p}}\\
    \nonumber &=  \xp_g \xp_{\tilde{g}} \Bigg[ \bigg\| \sum_{j \in \ints^d} \Big(  \sum_{n_{\tilde{J}}\in (\ints^d)^{k-1}} h_{n_Jn_An_B} \mathcal{L}\big(g_{n_{\tilde{J}}}^{\iota_{\tilde{J}}}\big) \Big)\tilde{g}_{n_j}\bigg\| _{n_A \rightarrow n_B}^p \Bigg]  ^{\frac{1}{p}} \\
    &\leq  C_1 \sqrt{p\log N}\!\! \max_{X \cupdot Y = \{j\}} \Bigg\{ \xp_g \Bigg[ \bigg\|\sum_{n_{\tilde{J}}\in (\ints^d)^{k-1}} \!\!\!\! h_{n_Jn_An_B} \mathcal{L}\big(g_{n_{\tilde{J}}}^{\iota_{\tilde{J}}}\big)  \bigg\|^p _{n_An_X  \rightarrow n_Bn_Y} \!\Bigg] ^\frac{1}{p} \Bigg\}, \label{eq: Main proof 1}
\end{align}
where $C_1$ is a constant in terms of $|A|$, $|B|$, and $d$. By the inductive hypothesis, the argument in the maximum in \eqref{eq: Main proof 1} can be estimated by  
\begin{align*}
    &\xp \bigg[ \Big\| \sum_{n_{\tilde{J}} \in (\ints^d)^{k-1}} \!\!h_{n_Jn_An_B} \el(g_{n_{\tilde{J}}}^{\iota_{\tilde{J}}} )  \Big\|^p _{n_An_X  \rightarrow n_Bn_Y}\! \bigg] ^\frac{1}{p} 
    \\
    &\leq C_2 p^{\frac{k-1}{2}} (\log N)^{\frac{k-1}{2}} \max_{Z \cupdot W = {\tilde{J}}} \|h\|_{n_An_Xn_Z \rightarrow n_Bn_Yn_W},
\end{align*}
where $C_2$ is a constant dependent on $|A|$, $|B|$, $d$, and $k$.
Thus, we have that \eqref{eq: Main proof 1} is bounded by 
\begin{align*}
    \eqref{eq: Main proof 1} & \leq C_1 \ p^{\frac{1}{2}} (\log N)^{\frac{1}{2}} \max_{X \cupdot Y = \{j\}} \left\{ C_{2} p^{\frac{k-1}{2}} (\log N)^{\frac{k-1}{2}} \max_{Z \cupdot W = {\tilde{J}}} \|h\|_{n_An_Xn_Z \rightarrow n_Bn_Yn_W} \right\} \\
    &= C_1 C_2 p^{\frac{k}{2}}  (\log N)^{\frac{k}{2}} \max_{X \cupdot Y = \{j\}, Z \cupdot W = {\tilde{J}}} \left\{ \|h\|_{n_An_Xn_Z \rightarrow n_Bn_Yn_W} \right\} \\
    &= C_1 C_2 p^{\frac{k}{2}}   (\log N)^{\frac{k}{2}} \max_{X \cupdot Y = J} \left\{ \|h\|_{n_An_X \rightarrow n_Bn_Y} \right\}, 
\end{align*}
where the last equality follows by observing that, for every pair of partitions $X$, $Y$ of $\{j\}$ and $W$, $Z$ of $\tilde{J}= J \backslash \{j\}$, the sets $X \cup Z$, $Y \cup W$ form a partition of $\{j\} \cup \tilde{J} = J$. 
We now have a bound on each term in the sum (that is independent of the particular $j \in J$), so we may rewrite \eqref{eq: Main Proof Induction} as
\begin{align*}
    & \sum_{j \in J} \frac{\pi}{2}   \xp \bigg[ \Big\| \sum_{n_J\in (\ints^d)^k} h_{n_Jn_An_B}   \tilde{g}_{n_j} \mathcal{L}(g_{n_{J\backslash j}}^{\iota_{J\backslash j}}) \Big\| _{n_A \rightarrow n_B}^p  \bigg] ^{\frac{1}{p}} \\&\leq  \ \sum_{j \in J} \tfrac{\pi}{2} C_1 C_2\  p^{\frac{k}{2}}  (\log N)^{\frac{k}{2}} \max_{X \cupdot Y = J} \left\{ \|h\|_{n_An_X \rightarrow n_Bn_Y} \right\} \\
     &= \ C p^{\frac{k}{2}}  (\log N)^{\frac{k}{2}} \max_{X \cupdot Y = J} \left\{ \|h\|_{n_An_X \rightarrow n_Bn_Y} \right\},
\end{align*}
where we set $C := k \cdot \frac{\pi}{2}\cdot C_1\cdot C_2$, which is a constant depending only on $|A|$, $|B|$, $d$, and $k$, concluding the proof.
\end{proof}

\section*{Acknowledgments}
The author would like to thank Professor Bjoern Bringmann for suggesting this problem, for providing direction, and for his incredibly generous support throughout this endeavor. The author would also like to thank Professor Ramon van Handel for very helpful comments and pointing to references in the literature.

\bibliographystyle{amsplain}
\bibliography{references}

\end{document}